\documentclass[12pt,onecolumn]{DaniArxiv}

\usepackage{cite}      
\usepackage{graphicx}
\usepackage{graphics}
\usepackage{epsfig}
\usepackage{psfrag}   
\usepackage{subfigure} 
\usepackage{amsmath}
\usepackage{amssymb}
\usepackage{deleq}   
\usepackage{type1cm} 

\newtheorem{remark}{Remark}
\newtheorem{theorem}{Theorem}
\newtheorem{lemma}{Lemma}
\newtheorem{example}{Example}
\newtheorem{corollary}{Corollary}

\title{Finite Approximations of Switched Homogeneous Systems for Controller Synthesis}
\author{Danielle C.~Tarraf\footnote{D. C. Tarraf is with the Electrical and Computer Engineering Department at the Johns 
Hopkins University, Baltimore, MD (dtarraf@jhu.edu).}, Alexandre~Megretski and
       Munther A.~Dahleh\footnote{A. Megretski and M. A. Dahleh are with Electrical Engineering and Computer Science Department and
the Laboratory for Information and Decision Systems at the Massachusetts Institute of Technology, Cambridge, MA (ameg@mit.edu, dahleh@mit.edu).}}

\begin{document}

\maketitle

\begin{abstract}

We demonstrate the use of a new, control-oriented notion of finite state approximation for a particular class of hybrid systems.
Specifically, we consider the problem of designing a stabilizing binary output feedback
switching controller for a pair of unstable homogeneous second order systems.
The constructive approach presented in this note, 
in addition to yielding an explicit construction of a deterministic finite state approximate model of the hybrid plant,
allows us to efficiently establish a useable upper bound on the quality of approximation, 
and leads to a discrete optimization problem whose solution immediately provides a certifiably correct-by-design controller for the original system.
The resulting controller consists of a finite state observer for the plant and a corresponding full state feedback switching control law.  

\end{abstract}


\section{Introduction}

\subsection{Motivation and Overview}

Finite state approximations and abstractions of hybrid plants have been explored as a means for addressing 
instances in which the interaction between analog and discrete 
dynamics\footnote{``Analog" and ``discrete" here refer to whether the dynamics 
evolve in a continuous or a discrete state-space;
the independent variable, time, may be continuous or discrete in either case.}
leads to complex analysis 
and synthesis problems that cannot be adequately handled by traditional methods.
Some of the early work in this area explored the use of `qualitative' models, namely non-deterministic automata 
whose output behavior contains that of the original hybrid system \cite{JOUR:Lunze1994}.
In \cite{JOUR:RaiO'Y1998}, 
an approach for deriving non-deterministic finite state approximations of systems with quantized outputs was proposed
and used in conjunction with the supervisory control theory developed by Ramadge and Wonham \cite{JOUR:RamWon1989}
to design controllers meeting the desired specifications.
Another line of research, inspired by formal verification techniques in computer science,
explored the construction of deterministic discrete abstractions of hybrid systems \cite{JOUR:AlHeLP2000}.
Early work in this area focused on identifying classes of systems that admit finite bisimulation abstractions \cite{JOUR:HeKoPV1998,JOUR:LaPaSa2000}.
Having recognized that these classes are fairly limited,
more recent work has focused on finding discrete abstractions that are related to the original hybrid system by an approximate simulation or
bisimulation relation \cite{JOUR:GirPap2007, JOUR:Tabuad2008, JOUR:TaAmAP2008}.
The resulting controller design problem is then a two step procedure in which a finite state supervisory controller is first 
designed, and subsequently refined to yield a certified hybrid controller for the original plant \cite{BOOK:Tabuad2009}.

In this note, we demonstrate the use of a new\footnote{Preliminary versions of this work were presented in 
\cite{CONF:TaMeDa2004, CONF:TMD,BOOKCHAPTER:TMD}.},
 alternative notion of approximation that,
in addition to yielding a finite state deterministic approximate model of the original hybrid system,
allows one to efficiently establish a useable upper bound on the quality of approximation, 
and leads to a discrete optimization problem whose solution immediately provides a certifiably correct-by-design (finite state) controller for the original system.
The notion of finite input/output approximation employed in this note 
is thus potentially better suited to the purpose of {\it control design} than the existing methods
since it allows for a streamlined and efficient {\it synthesis procedure}.
Another promising development is that, 
in contrast to many of the existing methods and our early work on this problem,
we have succeeded in approximating {\it unstable} systems by finite state automata for the purpose of control design.

We demonstrate this approach in a particular setting, 
namely for the problem of designing a stabilizing, binary output feedback, 
switching controller for a pair of unstable, homogeneous, second order systems.
This problem was specifically chosen because several of its simpler formulations have been extensively studied and are well understood 
\cite{JOUR:DeBrPL2000, BOOK:Liberz}.
Necessary and sufficient conditions for stability of switched second order homogeneous systems were derived in \cite{JOUR:Filipp1980}.
A Lyapunov based approach for designing stabilizing full state feedback switching controllers for second order homogenous systems was proposed in \cite{JOUR:HolMar2003}.
While the results available for general homogeneous systems remain limited, the special case of switched linear systems has been more extensively studied.
For instance, in the linear full state feedback analog case,
the existence of a Hurwitz convex combination ($A_{eq}=\alpha A_1 + (1-\alpha)A_2$, $\alpha \in (0,1)$) of a given pair of unstable state matrices ($A_1$ and $A_2$) 
is known to be a necessary \cite{TR:Feron1996} and 
sufficient \cite{JOUR:WiPeDe1998} condition for the
existence of a quadratically stabilizing switching controller.
The switching controller can take on a variety of forms in this case (time-switched controller, hybrid controller, etc...).
In particular, when $A_{eq}$ has a real eigenvalue,
a quadratic switching surface and corresponding state dependent variable structure control law \cite{JOUR:HuGaHu1993} 
can be designed to stabilize the system \cite{JOUR:WiPeDe1998}.
In this case, once designed assuming full state feedback,
the stabilizing controller can be exactly {\it implemented} by transmitting {\it appropriately chosen} binary state information from the plant to the controller,
namely information that the state is currently in one or the other region of the state-space. 
Lyapunov based approaches have also been extended to time-sampled versions of the problem \cite{JOUR:DaRiIu2002,JOUR:LinAnt2009}, 
and to the output feedback case \cite{JOUR:ShWMWK2007}: 
Typically the output is assumed to be a linear function of the state, and the resulting controller is observer-based \cite{TR:Feron1996}. 
In contrast, the setting considered in this note where {\it fixed, binary} sensors
are used leads to a more difficult state estimation (and thus controller synthesis) 
problem\footnote{This setup 
differs significantly from the variable structure controller setup described above,
as the binary sensors are assumed to 
be fixed a priori with no regard to the system dynamics.
As such, there generally is a mismatch between the information provided by the sensors and the binary information needed to implement 
a variable structure switching law. 
The resulting state estimation problem is at the core of the design challenge here.}.
This case has been much less well studied, 
and has only begun to receive attention in the recent past
\cite{JOUR:BenGav2002, BOOKCHAPTER:M, JOUR:SaMeDa2008}.
Specifically, the use of finite state approximations in this context remains minimally explored.

The constructive controller synthesis approach can be summarized as follows:
First, the plant is approximated by a deterministic finite state machine and a usable bound on the quality of approximation
is efficiently established.
Next, a switching control law is designed to robustly stabilize the nominal finite state machine model in the presence of admissible approximation uncertainty. 
The resulting controller, which is thus certifiably correct-by-design, is then also a finite state system. 
Additionally, the procedure by which it is synthesized is computationally tractable in the sense that it is based on a \textit{finite} state nominal model. 

Practically, this design approach may be utilized in applications where very coarse sensors are used to keep operating cost, 
weight or power consumption low. 
Looking ahead, having a systematic procedure for designing control systems specifically for the case where the plant and controller 
interact through finite alphabets points to new paradigms for control over networks in which the amount of information to be encoded and
transmitted over communication channels is significantly reduced. 
Finally, this design procedure may be employed by artificially imposing coarse measurements on a switched system for which there is no other available 
or tractable approach for synthesizing controllers, though the relevant question of how to best impose a finite measurement quantization for 
the control objective of interest is not addressed in this paper.

\subsection{Organization of the Paper}
The paper is organized as follows: 
The statement of the control design problem and an overview of the controller design procedure are described in Section \ref{Sec:StatementandOverview}. 
The algorithms proposed for constructing a finite state approximation of the plant, 
and for computing an a-posteriori bound on the resulting approximation error are presented in Section \ref{Sec:Approximation}. 
Design of the robust stabilizing switching law and the structure of the resulting switching controller are addressed in Section \ref{sec:Design}. 
Illustrative examples are presented in Section \ref{sec:Examples}. 
The paper concludes in Section \ref{sec:Conclusions}, where recommendations for future work are given.   

\subsection{Notation}
The notation used in the paper is fairly standard: 
$\mathbb{Z}_+$, $\mathbb{R}$ and $\mathbb{R}_+$ denote the sets of non-negative integers, reals and non-negative reals, respectively. 
For real interval $I = [a,b)$, $|I| = |b-a|$ denotes the length of $I$.
Given a vector $v \in \mathbb{R}^n$, $v'$ denotes its transpose and $\| v \| = \sqrt{v' v}$ denotes its Euclidean norm. 
Given sets $A$ and $B$, $card(A)$ denotes the cardinality of $A$, $A \times B$ denotes their Cartesian product, $A^{\mathbb{Z}_+}$ denotes 
the set of all infinite sequences taking their values in set $A$ and (boldface) $\mathbf{a}$ denotes an element of $A^{\mathbb{Z}_+}$. 
Given a function $f: A \rightarrow B$ and a proper subset $A' \subset A$, $f|_{A'}$ denotes the restriction of $f$ to $A'$, that is the function 
$f|_{A'} : A' \rightarrow B$ defined by $f|_{A'}(a) = f(a)$ for $a \in A'$. 
Given two functions $f: A \rightarrow B$ and $g:B \rightarrow C$, 
their composition, denoted by $g \circ f$, is a function $g \circ f: A \rightarrow C$ defined by $g \circ f(a) = g(f(a))$. 
Given two real-valued functions $f: A \rightarrow \mathbb{R}$ and $g: A \rightarrow \mathbb{R}$, the notation $f \leq g$ 
is understood to mean that $f(a) \leq g(a)$ for all $a \in A$.

\section{Problem Statement \& Overview of the Design Procedure}
\label{Sec:StatementandOverview}

\subsection{Problem Statement}
\label{Ssec:ProblemStatement}

Consider a discrete-time plant $P$ described by
\begin{equation}
\label{eq:PlantState}
x(t+1) = f_{u(t)}(x(t)) 
\end{equation}
\begin{equation}
\label{eq:PlantSensorOutput}
y(t) = sign \big( c' x(t) \big) 
\end{equation}
\begin{equation}
\label{eq:PlantPerformanceOutput}
v(t) = log \Big( \frac{\| x(t+1) \|}{\| x(t) \|} \Big)
\end{equation}
where the time index $t \in \mathbb{Z}_+$, state $x(t) \in \mathbb{R}^2$, performance output $v(t) \in \mathbb{R}$, and control 
input $u$ is binary with $u(t) \in \mathcal{U} =\{0,1\}$. 
Sensor output $y$ is also taken to be binary, $y(t) \in \mathcal{Y} = \{-1,1\}$, with the understanding that when $c'x=0$, 
$y$ is taken to be $+1$ in one quadrant and $-1$ in the other.
Functions $f_{u} : \mathbb{R}^2 \rightarrow \mathbb{R}^2$ and vector $c \in \mathbb{R}^{2}$, $c \neq 0$, are given. 
The following conditions are assumed to hold for each $u \in \{0,1\}$:
\begin{enumerate}
\item $f_{u}$ is continuous.
\item $f_u$ is homogeneous with degree 1: That is, $f_{u}(\alpha x)= \alpha f_{u}(x)$ for all 
$\alpha \in \mathbb{R}$, $x \in \mathbb{R}^2$.
\end{enumerate}
In this setup, the effect of the control action is to pick a choice of system from a given pair and to hold that choice until the next 
time step, at which point a new measurement is available and a new choice of system is made and implemented.

The objective is to design a controller $K \subset \mathcal{Y}^{\mathbb{Z}_+} \times \mathcal{U}^{\mathbb{Z}_+}$ 
such that the closed loop system $(P,K)$ with output $v$ (Figure \ref{fig:DesignSetup}) satisfies the following performance 
objective for some $R > 0$, for any initial condition of plant $P$:
\begin{equation}
\label{eq:PerfObjective}
\sup_{T \geq 0} \sum_{t=0}^{T} \Big( v(t) + R \Big) < \infty 
\end{equation}
Satisfaction of this performance objective guarantees that the state of the closed loop system (globally) exponentially converges to the origin, 
on average, at a rate not less than $R$.

  \begin{figure}[htbp]
     \begin{center}
     \includegraphics[scale=0.5]{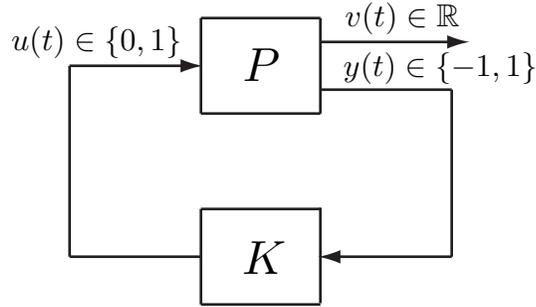}
     \caption{Closed loop system.}
     \label{fig:DesignSetup}
     \end{center}
   \end{figure}

\subsection{Overview of the Controller Design Procedure}
\label{ssec:SolutionOverview}

Design of the stabilizing controller is an iterative procedure, with each iteration consisting of a sequence of steps. 
First, the plant $P$ is approximated by the interconnection of a finite state machine $M$ and a complex system $\Delta$ representing the approximation 
error (Figure \ref{fig:Approximation}). 
Next, a meaningful and useable ``gain bound'' is established for the approximation error, system $\Delta$. 
Finally, an attempt is made to synthesize a feedback switching law for the nominal finite state model $M$, 
that is robust to the approximation error $\Delta$. 
If synthesis is successful, the resulting controller is guaranteed to globally exponentially stabilize the plant at some verified rate $R$. 
Otherwise if synthesis is unsuccessful, or if the verified rate is unsatisfactory, a more refined approximation 
(meaning a finite state machine with a larger number of states) is sought for the plant and the above process is repeated.

   \begin{figure}[thpb]
      \begin{center}
      \includegraphics[scale=0.5]{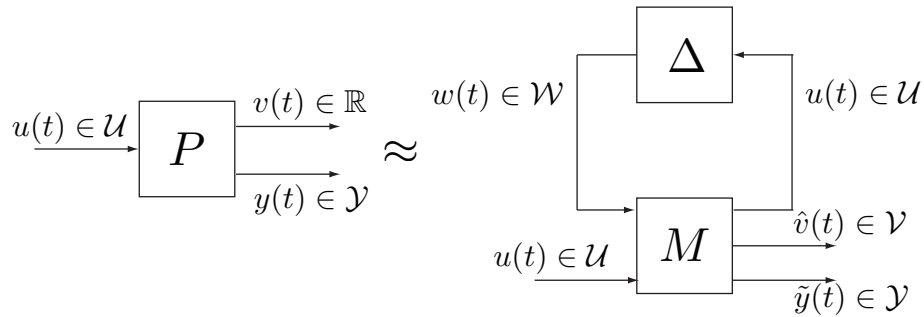}
      \caption{Plant $P$, its finite state machine approximation $M$, and the corresponding error system $\Delta$.}
      \label{fig:Approximation}
      \end{center}
   \end{figure}

\section{A Finite State Approximation of the Plant}
\label{Sec:Approximation}
 
A deterministic finite state machine (DFM for short) is understood to be a discrete-time dynamical system (independent variable $t \in \mathbb{Z}_+$) 
described by a state transition equation (\ref{eq:DFMstatetransition}) and an output equation (\ref{eq:DFMoutput})
\begin{eqnarray}
\label{eq:DFMstatetransition}
q(t+1) & = & f(q(t),u(t)) \\
\label{eq:DFMoutput}
y(t) & = & g(q(t,u(t)))
\end{eqnarray}
State $q(t) \in \mathcal{Q}$, input $u(t) \in \mathcal{U}$, and output $y(t) \in \mathcal{Y}$, 
where $\mathcal{Q}$, $\mathcal{U}$ and $\mathcal{Y}$ are \textit{finite} state, input alphabet and output alphabet sets, respectively.

\subsection{A Notion of Approximation} 

Given a plant $P$ with binary control input $u(t) \in \mathcal{U}=\{0,1\}$, binary sensor output $y(t) \in \mathcal{Y}=\{-1,1\}$ and upper bounded\footnote{It will
become clear in Section \ref{ssec:Mconstruction} that $v$ is indeed upper bounded for the class of systems under consideration.} 
performance output $v(t) \in \mathbb{R}$. 
Consider a system $M$ with inputs $u(t) \in \mathcal{U}$ and $w(t) \in \mathcal{W} = \{0,1\}$, 
and with outputs $u(t) \in \mathcal{U}$, $\tilde{y}(t) \in \mathcal{Y}$ and $\hat{v}(t) \in \mathcal{V}$, where $\mathcal{V}$ is a finite discrete subset of $\mathbb{R}$.
Suppose that $M$ has the internal structure shown in Figure \ref{fig:MStructure}, where $\hat{M}$ is a deterministic 
finite state machine. 
To ensure well-posedness, we require that there be no direct feedthrough from input $\tilde{y}$ to output $\hat{y}$ in $\hat{M}$:
In other words, $\hat{y}(t)$ is only allowed to be a function of the input $u(t)$ and the state of $\hat{M}$ at time $t$.
Memoryless system $\phi$ is defined by
\begin{displaymath}
\tilde{y} = 
\left\{ \begin{array} {ll}
\hat{y} & \mbox{  if  } w=0 \\
-\hat{y} & \mbox{  if  } w=1 
\end{array} \right.
\end{displaymath}
System $M$ is thus a deterministic finite state machine. 

Consider also the corresponding system $\Delta$ shown in Figure \ref{fig:DeltaSystem}. 
Memoryless system $\Psi$ is defined by
\begin{displaymath}
w = \left\{ \begin{array} {ll}
0 & \mbox{  if  } y = \hat{y} \\
1 & \mbox{  if  } y \neq \hat{y}
\end{array} \right.
\end{displaymath}
System $\Delta$ is not a deterministic finite state machine in general, unless plant $P$ is one.
Moreover, the existing requirement that there be no direct feedthrough from $y$ to $\hat{y}$ in $\hat{M}$ 
ensures that the outputs $y$ of $P$ and $\hat{y}$ of $\hat{M}$ cannot be trivially matched.

   \begin{figure}[thpb]
      \begin{center}
      \includegraphics[scale=0.5]{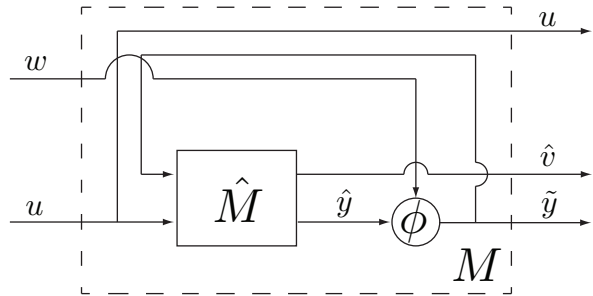}
      \caption{Internal structure of $M$, the deterministic finite state machine approximation of $P$.}
      \label{fig:MStructure}
      \end{center}
   \end{figure}

Now consider the interconnection, as in Figure \ref{fig:Approximation}, of $M$ and $\Delta$ with these particular structures, and suppose 
that the two copies of $\hat{M}$ are identically initialized. Regardless of the choice of $\hat{M}$, it can be seen by direct inspection that 
for arbitrary initial conditions of $P$ and for any input $\mathbf{u} \in \mathcal{U}^{\mathbb{Z}_+}$, the corresponding outputs 
$\mathbf{y} \in \mathcal{Y}^{\mathbb{Z}_+}$ of $P$ and $\tilde{\mathbf{y}} \in \mathcal{Y}^{\mathbb{Z}_+}$ of $(M,\Delta)$ are identical. 

Suppose that we can construct $\hat{M}$ such that the following two conditions are satisfied for any input 
$\mathbf{u} \in \mathcal{U}^{\mathbb{Z}_+}$ and any initial condition of $P$:
\begin{enumerate}
\item{} Outputs $\mathbf{\hat{v}} \in \mathcal{V}^{\mathbb{Z}_+}$ of the interconnection $(M,\Delta)$ and $\mathbf{v} \in \mathbb{R}^{\mathbb{Z}_+}$ 
of $P$ satisfy
\begin{equation}
\label{eq:vBound}
\hat{v}(t) \geq v(t), \mbox{  for all  } t \in \mathbb{Z}_+.
\end{equation}
\item{} For $\rho: \mathcal{U} \rightarrow \{1\}$ and $\mu:\mathcal{W} \rightarrow \mathbb{R}_+$ defined by $\mu(w)=w$, 
there exists a constant $\gamma \in (0,1)$ such that every feasible input/output pair of signals 
$(\mathbf{u},\mathbf{w}) \in \mathcal{U}^{\mathbb{Z}_+} \times \mathcal{W}^{\mathbb{Z}_+}$ of the 
error system $\Delta$ satisfies the gain condition:
\begin{equation}
\label{eq:gainofDelta}
\inf_{T \geq 0} \sum_{t=0}^{T} \gamma \rho(u(t)) - \mu(w(t)) > - \infty.
\end{equation}
\end{enumerate}
The resulting deterministic finite state machine $M$ is then said to be a \textit{finite state approximation} of $P$, and the corresponding 
system $\Delta$ is said to be the \textit{approximation error}. 
The following remarks aim to clarify the reasoning behind this approach to plant approximation.

   \begin{figure}[thpb]
      \begin{center}
      \includegraphics[scale=0.5]{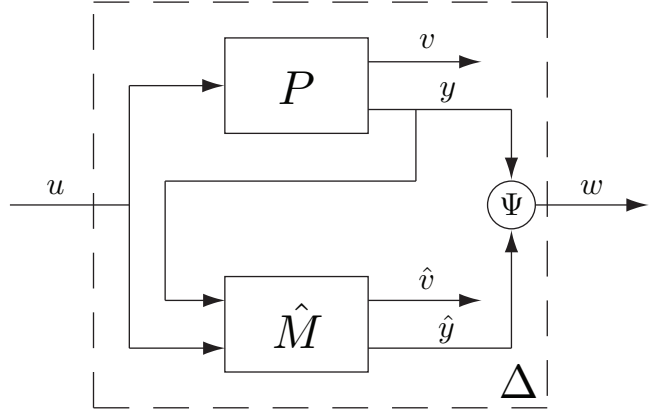}
      \caption{The error system $\Delta$ of $\hat{M}$ and $P$.}
      \label{fig:DeltaSystem}
      \end{center}
   \end{figure}

\begin{remark}
The structure of the approximation error $\Delta$ proposed in this setting is significantly different from that of the approximation error in 
the traditional stable LTI model reduction setting. 
In that setting, a stable lower order LTI model is considered to be a good approximation 
of the original stable LTI system if the outputs of the two systems are not too different, in a worst case sense, when driven side by side by the same input. 
The structure proposed here, where the output of the plant $P$ is fed back to $\hat{M}$, is needed because the original system is not stable: 
Two copies of $P$ initialized differently and driven side by side may end up with different outputs at every step. 
Thus, there is a need to explicitly 'estimate' the initial condition of $P$, otherwise there is no hope for satisfying gain condition (\ref{eq:gainofDelta}) for any choice of $\hat{M}$.
\end{remark}

\begin{remark}
The first condition characterizing a valid finite state approximation, namely the inequality in (\ref{eq:vBound}), is motivated by the control objective (\ref{eq:PerfObjective}) at hand.
In particular, if (\ref{eq:vBound}) holds and a controller $\varphi \subset \mathcal{Y}^{\mathbb{Z}_+} \times \mathcal{U}^{\mathbb{Z}_+}$ is designed to ensure that the closed loop system $(M,\Delta,\varphi)$ satisfies the auxiliary 
performance objective 
\begin{equation}
\label{eq:AuxPerfObject}
\sup_{T \geq 0} \sum_{t=0}^{T} \hat{v}(t) + R < \infty,
\end{equation}
then a corresponding controller implemented in feedback with plant $P$ is guaranteed to satisfy performance objective (\ref{eq:PerfObjective}): 
We will return to address this statement in more detail in Section \ref{sec:Design}.
Upper boundedness of signal $v$ ensures that (\ref{eq:vBound}) can always be satisfied for some appropriate finite choice of $\mathcal{V}$. 
Intuitively, a better approximation in which the instantaneous difference between $v$ and $\hat{v}$ is smaller 
is desirable and leads to a less conservative controller design.
\end{remark} 

\begin{remark}
The second condition characterizing a valid finite state approximation, namely the 'gain condition' describing the approximation error $\Delta$, 
is compatible with the framework and the tools for robustness analysis of systems over finite alphabets developed in \cite{JOUR:TMD}. 
The choice of functions $\rho$ and $\mu$ proposed here are specific to the control objective at hand.
In particular, the smallest value of $\gamma$ for which (\ref{eq:gainofDelta}) holds (the 'gain' of $\Delta$) represents 
the fraction of time (computed over an infinite window) that the outputs of $P$ and $\hat{M}$ disagree in the worst-case scenario.  
A smaller value of $\gamma$ is thus desirable and indicative of a better approximation. 
\end{remark}

In the following sections, a constructive procedure for generating a viable $\hat{M}$ and a computationally efficient algorithm for computing an 
a-posteriori upper bound on the gain of the resulting error system $\Delta$ are proposed.

\subsection{Construction of the Nominal Model}
\label{ssec:Mconstruction}

The approach proposed for constructing a nominal finite state model to approximate plant $P$ takes advantage of the dynamical 
properties specific to homogeneous systems, evident after a coordinate transformation; 
interested readers are referred to \cite{THESIS:Ta} for an overview of other potential approaches.
Let 
\begin{displaymath}
x= \left[ \begin{array}{c} x_1 \\ x_2 \end{array} \right], \textrm{     }
c= \left[ \begin{array}{c} c_1 \\ c_2 \end{array} \right], \textrm{     }
f_u (.) = \left[ \begin{array}{c} f_{u1}(.) \\ f_{u2}(.) \end{array} \right]
\end{displaymath}

In a polar coordinate system where $r = \sqrt{ x_1^{2} + x_{2}^{2}}$ and $\displaystyle \theta = \tan^{-1} \left( \frac{x_2}{x_1} \right)$, 
the dynamics of system $P$ described in (\ref{eq:PlantState}), (\ref{eq:PlantSensorOutput}) and (\ref{eq:PlantPerformanceOutput}) are given by:
\begin{displaymath}
r(t+1) = r(t) \Big\| f_{u(t)} \big(\beta(\theta(t)) \big) \Big\|
\end{displaymath}
\begin{equation}
\label{eq:SystemPpolar}
\theta(t+1) = \tan^{-1} \Big( \frac{ f_{u(t)2}\big( \beta(\theta(t)) \big) } { f_{u(t)1}\big( \beta(\theta(t)) \big) } \Big)
\end{equation}
\begin{equation}
\label{eq:SystemPoutputy}
y(t) = sign \Big( c' \beta(\theta(t)) \Big)
\end{equation}
\begin{equation}
\label{eq:SystemPoutputv}
v(t)=log \Big( \Big\| f_{u(t)} \big(\beta(\theta(t)) \big) \Big\| \Big)
\end{equation}
where 
\begin{displaymath}
\beta(\theta) = \left[ \begin{array} {c} \cos{\theta} \\ \sin{\theta} \end{array} \right]. 
\end{displaymath}

This coordinate transformation highlights two important properties of the class of systems under consideration:
\begin{enumerate}
\item The evolutions of the angular coordinate $\theta$ and of both outputs of $P$ are independent of the radial coordinate $r$:
The state of the system relevant to the stabilization problem at hand effectively evolves on the unit circle. 
\item For brevity of notation in the following argument, denote the composition $\beta \circ \theta $ simply by $\beta$. 
For each $t \in \mathbb{Z}_+$, we have $\beta(t) \in [-1,1] \times [-1,1] = R_1$, 
$f_{u(t)}(\beta(t)) \in f_{u(t)}(R_1) = R_2^{u(t)}$,
where $R_2^{0}$ and $R_2^{1}$ are again compact in $\mathbb{R}^2$ since each $f_u$ is continuous.
Thus we have 
\begin{displaymath}
\|f_{u(t)}(\beta(t))\| \leq \max_{x \in R_2^u, \\ u \in \{0,1\}} \|x\| = M
\end{displaymath}
and
\begin{displaymath}
v(t) = log \Big( \|f_{u(t)}(\beta(t))\| \Big) \leq  logM, \textrm{    for all } t \in \mathbb{Z}_+.
\end{displaymath}
It thus follows that $v$ is upper bounded.
\end{enumerate}
Compactness of the effective state set and upper boundedness of the performance output $v$ are instrumental in ensuring that 
this class of plants is amenable to a finite state approximation. 

$\hat{M}$ is constructed by partitioning the unit circle into a collection of intervals and defining the potential states of $\hat{M}$ to correspond 
to the intervals and the unions of adjacent intervals.
This partition generally need not consist of equal length intervals; however, it should be chosen such that the outputs associated with all the (analog) states 
of $P$ whose angular coordinates lie in any given interval are identical, in order to avoid introducing unnecessary uncertainty.

In particular, consider a partition of the unit circle consisting of intervals $I_{1}, \ldots, I_{2n}$ where $ I_{i} = [\alpha_{i},\alpha_{i+1} )$ 
for some sequence of angles $\alpha_{1} < \ldots < \alpha_{2n+1}$ satisfying:
\begin{displaymath}
\alpha_{1} = tan^{-1} \Big( -\frac{c_1}{c_2} \Big), \textrm{     } \alpha_1 \in [0,\pi)
\end{displaymath}
\begin{displaymath}
\alpha_{n+1} = \alpha_{1} + \pi
\end{displaymath}
\begin{displaymath}
\alpha_{2n+1} = \alpha_{1} + 2 \pi
\end{displaymath}
The number and choice of angles is a design parameter here.

Construct a set $\hat{\mathcal{Q}}$ of intervals on the unit circle:
\begin{displaymath}
\hat{\mathcal{Q}} = S_1 \cup S_2 \cup \hdots \cup S_{2n}
\end{displaymath}
where $S_k$ is the set of all "$k$ adjacent intervals". That is:
\begin{eqnarray*}
S_1 & = & \{I_1,I_2, \hdots, I_{2n}\} \\
S_2 & = & \{I_1 \cup I_2, I_2 \cup I_3, \hdots, I_{2n} \cup I_1 \} \\
S_3 & = & \{I_1 \cup I_2 \cup I_3, I_2 \cup I_3 \cup I_4, \hdots, I_{2n} \cup I_1 \cup I_2 \} \\
& & \vdots \\
S_{2n} & = & \{I_1 \cup I_2 \cup \hdots \cup I_{2n} \}
\end{eqnarray*}
This set $\hat{\mathcal{Q}}$, consisting of $2n(2n-1)+1$ distinct elements, is the set of all \textit{potential states} of $\hat{M}$.
It will become clear shortly that continuity of $f_u$ rules out the union of two non-adjacent quantization intervals from being a potential state of $\hat{M}$ 
in this setup. 

For $q \in \hat{\mathcal{Q}}$, let
\begin{displaymath}
\mathcal{P}_1(q) = \left\{ \begin{array}{cl}
q \cap [\alpha_1,\alpha_{n+1}) & \textrm{     if     } sign(c'\beta(\alpha_2))=1\\
q \cap [\alpha_{n+1},\alpha_{2n+1}) & \textrm{     otherwise} 
\end{array} \right.
\end{displaymath}
and 
\begin{displaymath}
\mathcal{P}_{-1}(q) = \left\{ \begin{array}{cl}
q \cap [\alpha_1,\alpha_{n+1}) & \textrm{     if     } sign(c'\beta(\alpha_2))=-1\\
q \cap [\alpha_{n+1},\alpha_{2n+1}) & \textrm{     otherwise} 
\end{array} \right.
\end{displaymath}
Note that $\mathcal{P}_1(q) \in \hat{Q}$ and $\mathcal{P}_{-1}(q) \in \hat{Q}$, by construction. 
For $u \in \mathcal{U}$, $q \in \hat{\mathcal{Q}}$, let
\begin{displaymath}
q_{+}^{u} = tan^{-1}\Big( \frac{f_{u2}(\beta(q))}{f_{u1}(\beta(q))} \Big)
\end{displaymath}
and
\begin{displaymath}
\mathcal{I}_{q}^{u} = \Big\{ i \in \{1,\hdots,2n\} \Big|  q_+^{u} \cap [\alpha_i,\alpha_{i+1}) \neq \emptyset \Big\}.
\end{displaymath}

\begin{remark}
It follows from the continuity of $f_u$ that for any choice of $u \in \mathcal{U}$ and $q \in \hat{\mathcal{Q}}$, 
$q_+^{u}$ is a single connected interval. Thus 
\begin{displaymath}
\bigcup_{i \in I_{q}^{u}} [ \alpha_i, \alpha_{i+1} )
\end{displaymath}
is an interval of the form $[\alpha_j,\alpha_k)$, in other words, an element of $\hat{\mathcal{Q}}$.
\end{remark}

The dynamics of $\hat{M}$ are then given by
\begin{eqnarray}
\label{eq:Mhatpotential}
\nonumber
q(t+1) & = & \hat{f}(q(t),u(t),\tilde{y}(t)) \\ 
\hat{y}(t) & = & \hat{g} (q(t)) \\ \nonumber
\hat{v}(t) & = & \hat{h} (q(t),u(t))
\end{eqnarray}
with $\hat{f} : \hat{\mathcal{Q}} \times \mathcal{U} \times \mathcal{Y} \rightarrow \hat{\mathcal{Q}}$, 
$\hat{g} : \hat{\mathcal{Q}} \rightarrow \mathcal{Y}$ and $\hat{h} : \hat{\mathcal{Q}} \times \mathcal{U} \rightarrow \mathcal{V}$ defined by
\begin{eqnarray*} 
\hat{f}(q,u,\tilde{y}) & = & \bigcup_{i \in \mathcal{I}_{\mathcal{P}_{\tilde{y}}(q)}^{u}}  [\alpha_i,\alpha_{i+1}) \\
\hat{g}(q) & = & \left\{ \begin{array}{lll}
1 & \textrm{   } & \textrm{if  } q = \mathcal{P}_1(q) \textrm{    or    } |\mathcal{P}_1(q)| \geq |\mathcal{P}_{-1}(q)|\\
-1 & \textrm{   } & \textrm{if  } q = \mathcal{P}_{-1}(q) \textrm{    or    } |\mathcal{P}_1(q)| < |\mathcal{P}_{-1}(q)|\\
\end{array} \right. \\
\hat{h}(q,u) & = & \sup_{\theta \in q} log \Big(\| f_{u} \big( \beta(\theta) \big) \| \Big)
\end{eqnarray*}

In particular, let $q_o$ denote the potential state corresponding to the whole unit circle, that is $q_o = [\alpha_1, \alpha_{2n+1} )$. 
It is shown in the following Lemma that condition (\ref{eq:vBound}) is satisfied provided $\hat{M}$ is initialized to this state.
\begin{lemma}
\label{lemma:MhatProperty}
Consider a plant $P$ and let $\hat{M}$ be the corresponding finite state machine defined by (\ref{eq:Mhatpotential}) for some choice of 
integer $n>0$ and (admissible) angles $\alpha_{1}$, $\hdots$ , $\alpha_{2n+1}$.
If $q(0)=q_o$, then for any input $\mathbf{u} \in \mathcal{U}^{\mathbb{Z}_+}$ and for any initial condition of $P$, the outputs 
$\mathbf{\hat{v}} \in \mathcal{V}^{\mathbb{Z}_+}$ of $(M,\Delta)$ and 
$\mathbf{v} \in \mathbb{R}^{\mathbb{Z}_+}$ of $P$ satisfy (\ref{eq:vBound}).
\end{lemma} 

\begin{proof} Let $\theta(t)$ and $q(t)$ denote the states of $P$ and $\hat{M}$, respectively, at time $t$. It follows from the 
construction of $\hat{M}$ that:
\begin{enumerate}
\item{} $\theta(t) \in q(t) \Rightarrow \hat{v}(t) \geq v(t)$
\item{} $\theta(t) \in q(t) \Rightarrow \theta(t+1) \in q(t+1)$
\end{enumerate}
We have $\theta(0) \in [0,2\pi) = [\alpha_1, \alpha_1 + 2 \pi) = [\alpha_1,\alpha_{2n+1})$. When $\hat{M}$ is initialized to $q(0)=q_o$, $\theta(0) \in q(0)$, hence $\hat{v}(0) \geq v(0)$ and $\theta(1) \in q(1)$.
The statement of the Lemma thus follows by induction on $t$.
\end{proof}

Since the initial state of $\hat{M}$ will be fixed to $q_o$, the \textit{actual} states of $\hat{M}$ consist of those states in 
$\hat{\mathcal{Q}}$ that are reachable from $q_o$. 
The problem of computing the reachable subset $\mathcal{Q}$ can be recast as either one of two well-studied problems:
(i) A one-to-all network shortest path problem, which can be efficiently solved (polynomial time in $n$) using any of the available shortest path algorithms 
(Dijkstra's, Bellman-Ford, $\hdots$) \cite{BOOK:AhMaOr}, 
or (ii) the problem of computing the accessible states of an automaton, which can be efficiently solved by constructing the transition tree of the automaton \cite{BOOK:Lawson}.
Thus any of a standard collection of algorithms can be used to compute $\mathcal{Q}$; 
interested readers are referred to the above two references for details of the various algorithms.

The dynamics of $\hat{M}$ are thus given by:
\begin{eqnarray}
\label{eq:Mhat}
\nonumber
q(t+1) & = & f(q(t),u(t),\tilde{y}(t)) \\ 
\hat{y}(t) & = & g (q(t)) \\ \nonumber
\hat{v}(t) & = & h (q(t),u(t))
\end{eqnarray}
where the state transition function and output functions are given by $f  = \hat{f}|_{\mathcal{Q} \times \mathcal{U} \times \mathcal{Y}}$ and
$g = \hat{g}|_{\mathcal{Q}} $, $h = \hat{h}|_{\mathcal{Q} \times \mathcal{U}}$, respectively.

\subsection{Description of the Approximation Error}

In this section, a procedure for computing an upper bound on the gain of the error system $\Delta$ associated with a given plant $P$ 
and a corresponding nominal finite state model $M$, constructed as described in the previous section, is presented. 
As before, $\mathcal{Q}$ refers to the actual states of $\hat{M}$.

\begin{lemma}
\label{thm:GammaBound}
If there exists a function $V : \mathcal{Q} \rightarrow \mathbb{R}$ and a $\gamma \in (0,1)$ such that
\begin{equation}
\label{eq:VgammaBound}
V(f(q,u,y)) - V(q) \leq \gamma \rho(u) - d(q)
\end{equation}
holds for all $q \in \mathcal{Q}$, $u \in \mathcal{U}$ and $y \in \mathcal{Y}$, where $d: \mathcal{Q} \rightarrow \{0,1\}$ 
defined by: 
\begin{displaymath}
d(q) = \left\{ \begin{array}{ll}
0 & \textrm{     if      } q=\mathcal{P}_{1}(q) \textrm{      or     } q = \mathcal{P}_{-1}(q) \\
1 & \textrm{     otherwise}
\end{array} \right. 
\end{displaymath}
then the error system $\Delta$ with $\hat{M}$ is initialized to $q_o$ satisfies (\ref{eq:gainofDelta}) for that choice of $\gamma$.
\end{lemma}

\begin{proof} By summing up (\ref{eq:VgammaBound}) along any state trajectory of $\hat{M}$ from $t=0$ to $t=T$, we get:
\begin{eqnarray} \nonumber
\sum_{t=0}^{T} \gamma \rho(u(t)) - d(q(t)) & \geq & V(q(T)) - V(q(0))\\ \nonumber
& \geq & \min_{q_1,q_2} \Big( V(q_1)-V(q_2) \Big)
\end{eqnarray}
Hence, we have:
\begin{displaymath}
\inf_{T \geq 0} \sum_{t=0}^T \gamma \rho(u(t)) - d(q(t)) \geq \min_{q_1,q_2} \Big( V(q_1)-V(q_2) \Big) > -\infty
\end{displaymath}
It follows from Lemma \ref{lemma:MhatProperty} that when $\hat{M}$ is initialized to $q(0)=q_o$, we have $\theta(t) \in q(t)$ 
for all $t$, where $\theta$ and $q$ are the states of $P$ and $\hat{M}$, respectively. 
Thus, when $q(t) = \mathcal{P}_1(q(t))$ or $q(t) = \mathcal{P}_{-1}(q(t))$,  $y(t)=\hat{y}(t)$ and $w(t)=0$.
Otherwise $w(t) \in \{0,1\}$. 
Hence $w(t) \leq d(q(t))$ for all $t$ and all feasible input/output signal pairs of $\Delta$ satisfy (\ref{eq:gainofDelta}).
\end{proof}

An upper bound for the gain of $\Delta$ can thus be computed by solving a linear program in which we minimize $\gamma$ such that 
(\ref{eq:VgammaBound}) holds for all $q \in \mathcal{Q}$, $u \in \mathcal{U}$ and $y \in \mathcal{Y}$. 
This linear program has $N+1$ decision variables and $4N$ inequality constraints, where $N=card(\mathcal{Q})$.

\begin{remark}
Recall that the approximation error $\Delta$ is a \textit{complex} system with both continuous and discrete states. 
Thus the appeal of this approach to computing a gain bound for $\Delta$ is its simplicity and its computational efficiency. 
The downside of this approach is that it results in conservative gain bounds, for two reasons:
\begin{enumerate}
\item It inherently assumes that an error occurs every single time it can.
\item It assumes that all signal pairs $(\mathbf{u},\mathbf{y}) \in \mathcal{U}^{\mathbb{Z}_+} \times \mathcal{Y}^{\mathbb{Z}_+}$ 
are valid input sequences for $\hat{M}$, which is not the case since $\mathbf{y}$ is an output of $P$ corresponding to $\mathbf{u}$.
\end{enumerate}
Thus, in using this approach, we are in effect trading off simplicity and efficiency versus conservatism.
\end{remark}

\section{Controller Design}
\label{sec:Design}

\subsection{Two Related Synthesis Problems}
\label{sec:SynthesisProblems}

   \begin{figure}[thpb]
      \begin{center}
      \includegraphics[scale=0.5]{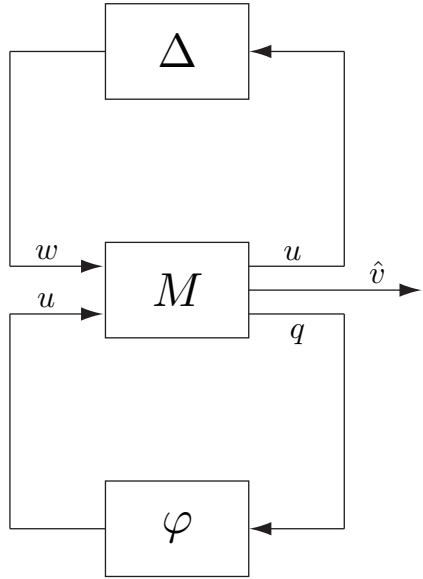}
      \caption{Robust full state feedback control design setup}
      \label{fig:RobustSetup}
      \end{center}
   \end{figure}

Consider the following two controller synthesis problems:

\begin{itemize}
\item[] \underline{\textbf{Problem 1}:} Given a plant $P$ as in (\ref{eq:PlantState}), (\ref{eq:PlantSensorOutput}) and (\ref{eq:PlantPerformanceOutput}), 
design a controller $K \subset \mathcal{Y}^{\mathbb{Z}_+} \times \mathcal{U}^{\mathbb{Z}_+}$ such that the feedback interconnection $(P,K)$ 
shown in Figure \ref{fig:DesignSetup} satisfies the performance objective
$$
\sup_{T \geq 0} \sum_{t=0}^{T} \Big( v(t) + R \Big) < \infty 
\reqno{eq:PerfObjective}
$$
for some $R>0$, for any initial condition of $P$.

\item[] \underline{\textbf{Problem 2}:} Given a plant $P$ as in (\ref{eq:PlantState}), (\ref{eq:PlantSensorOutput}) and (\ref{eq:PlantPerformanceOutput}) 
and a finite state approximation $M$ of $P$ constructed as described in Section \ref{ssec:Mconstruction} with 
the corresponding error system satisfying gain condition (\ref{eq:gainofDelta}) for some $\gamma \in (0,1)$. 
Design a full state feedback control law $\varphi: \mathcal{Q} \rightarrow \mathcal{U}$ such that the interconnection 
$(M,\Delta,\varphi)$ shown in Figure \ref{fig:RobustSetup} satisfies the auxiliary robust performance objective
$$
\sup_{T \geq 0} \sum_{t=0}^{T} \Big( \hat{v}(t) + R \Big) < \infty
\reqno{eq:AuxPerfObject}
$$
for some $R>0$, for all systems $\Delta$ satisfying gain condition (\ref{eq:gainofDelta}).

\end{itemize}

Problem 1 is the original problem of interest. 
Satisfaction of performance objective (\ref{eq:PerfObjective}) guarantees that the state of the closed loop system globally exponentially converges to the origin, 
on average, at a rate not less than $R$.
To see that, note that this performance objective can be equivalently re-written as
\begin{displaymath}
\| x(t) \| \leq k(x(0)) 10^{-Rt} \|x(0)\|, \textrm{       } \forall t \geq 0, x(0) \in \mathbb{R}^2
\end{displaymath}
where $\displaystyle k(x(0)) = 10^{S(x(0)) - R}$ with
\begin{displaymath}
S(x(0) = \sup_{T>0} \sum_{t=0}^{T} (v(t) + R)
\end{displaymath}
when the plant is initialized to $x(0)$.
The largest value of $R$ for which (\ref{eq:AuxPerfObject}) holds is then the guaranteed rate of 
exponential\footnote{The notion of exponential convergence considered here is slightly different than that of global exponential convergence defined in standard textbooks 
such as \cite{BOOK:Vidyasagar2002,BOOK:Khalil2002}; in particular, we are not requiring the term $k(x(0))$ to be uniformly bounded.} convergence; 
the actual rate of convergence may be (significantly) better.

Problem 2 is of interest as its solution provides a solution for Problem 1. 
In particular, let $\varphi$ be a full state feedback controller such that $(M,\Delta,\varphi)$ satisfies (\ref{eq:AuxPerfObject}) for some $R > 0$, for all admissible $\Delta$. 
Recall that by construction, the outputs $\hat{v}$ of $(M,\Delta)$ and $v$ of $P$ satisfy (\ref{eq:vBound})
whenever the same input $\mathbf{u} \in \mathcal{U}^{\mathbb{Z}_+}$ drives the plant $P$ and the interconnection $(M,\Delta)$. 
Thus, to ensure that (\ref{eq:PerfObjective}) holds for interconnection $(P,K)$ whenever (\ref{eq:AuxPerfObject}) 
holds, it is sufficient to ensure that the controller $K$ connected in feedback with plant $P$ is identical to the subsystem with input $y$ 
and output $u$ in the interconnection $(M,\Delta,\varphi)$. The structure of the resulting controller $K$ is shown in Figure \ref{fig:Implementation}: 
$K$ thus consists of $\hat{M}$, a deterministic finite state machine "observer" for the plant and $\varphi$, a corresponding full state feedback control law.

 \begin{figure}[thpb]
      \begin{center}
      \includegraphics[scale=0.5]{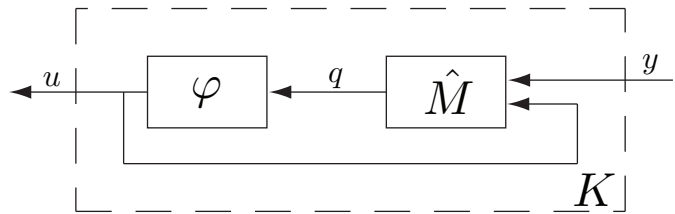}
      \caption{The finite state stabilizing controller $K$}
      \label{fig:Implementation}
      \end{center}
   \end{figure}

\subsection{Design of a Full State Feedback Robust Switching Law}
\label{sec:ControllerDesign}

The following problem is addressed in this section: 
Given a deterministic finite state machine 
$M \subset (\mathcal{W}^{\mathbb{Z}_+} \times \mathcal{U}^{\mathbb{Z}_+}) \times (\mathcal{U}^{\mathbb{Z}_+} \times \mathcal{V}^{\mathbb{Z}_+} \times \mathcal{Y}^{\mathbb{Z}_+})$ 
with state set $\mathcal{Q}$, a scalar $\gamma_o \in (0,1)$, and an uncertainty class 
\begin{displaymath}
\mathbf{\Delta}_{\gamma_o} = \Big\{ \Delta \subset \mathcal{U}^{\mathbb{Z}_+} \times \mathcal{W}^{\mathbb{Z}_+} \Big| \Delta \textrm{   satisfies (\ref{eq:gainofDelta}) for    } \gamma=\gamma_o \Big\}.
\end{displaymath}
Design a robust switching law $\varphi: \mathcal{Q} \rightarrow \mathcal{U}$ such that the closed loop system $(M,\Delta,\varphi)$ shown in Figure \ref{fig:RobustSetup}
satisfies the auxiliary robust performance objective
$$
\sup_{T \geq 0} \sum_{t=0}^{T} \Big( \hat{v}(t) + R  \Big) < \infty
\reqno{eq:AuxPerfObject}
$$
for some $R >0$, for all admissible uncertainty $\Delta \in \mathbf{\Delta}_{\gamma_o}$. 

The robust switching law $\varphi$ will be designed using dynamic programming techniques and a ``small gain'' argument. 
Consider the feedback interconnection of two systems $S$ and $\Delta$ as in Figure \ref{fig:SGT}. The 'Small Gain Theorem' (Theorem \ref{thm:SGT}) and Corollary \ref{cor:Scales}, 
presented here without proof, are directly adapted from Theorem 1 and Remark 4, respectively, in \cite{JOUR:TMD}. 

\begin{theorem}
\label{thm:SGT}
\textit{(A 'Small Gain' Theorem)}  Suppose that $S$ satisfies
\begin{equation}
\label{eq:SGT1}
\inf_{T \geq 0} \sum_{t=0}^{T} \rho_{S}(r(t),w(t)) - \mu_{S}(\hat{v}(t),u(t)) > -\infty
\end{equation}
for some $\rho_{S} : \mathcal{R} \times \mathcal{W} \rightarrow \mathbb{R}$ and $\mu_{S} : \hat{\mathcal{V}} \times \mathcal{U} \rightarrow \mathbb{R}$, 
where $\mathcal{R}$, $\mathcal{W}$, $\hat{\mathcal{V}}$ and $\mathcal{U}$ are finite sets, and that $\Delta$ satisfies
\begin{equation}
\label{eq:SGT2}
\inf_{T \geq 0} \sum_{t=0}^{T} \rho_{\Delta}(u(t)) - \mu_{\Delta}(w(t)) > -\infty
\end{equation}
for some $\rho_{\Delta}: \mathcal{U} \rightarrow \mathbb{R}$ and $\mu_{\Delta} : \mathcal{W} \rightarrow \mathbb{R}$. Then the 
interconnected system $(S,\Delta)$ with input $r$ and output $\hat{v}$ satisfies
\begin{equation}
\label{eq:SGT3}
\inf_{T \geq 0} \sum_{t=0}^{T} \rho(r(t)) - \mu(\hat{v}(t)) > -\infty
\end{equation}
for $\rho: \mathcal{R} \rightarrow \mathbb{R}$ and $\mu: \hat{\mathcal{V}} \rightarrow \mathbb{R}$ defined by
\begin{displaymath}
\label{eq:SGT4}
\rho(r) = \max_{w \in \mathcal{W}} \{ \rho_{S}(r,w) - \mu_{\Delta}(w) \}
\end{displaymath}
\begin{displaymath}
\label{eq:SGT5}
\mu(\hat{v}) = \min_{u \in \mathcal{U}} \{ \mu_{S}(\hat{v},u) - \rho_{\Delta}(u) \}
\end{displaymath}
\end{theorem} 

   \begin{figure}[thpb]
      \begin{center}
      \includegraphics[scale=0.5]{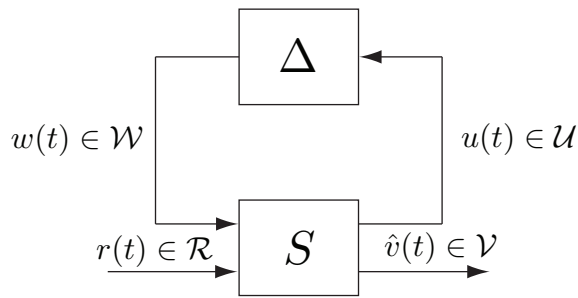}
      \caption{Setup for the 'Small Gain' Theorem.}
      \label{fig:SGT}
      \end{center}
   \end{figure}

\begin{corollary}
\label{cor:Scales}
The interconnected system satisfies (\ref{eq:SGT3}) for $\rho : \mathcal{R} \rightarrow \mathbb{R}$ and $\mu: \hat{\mathcal{V}} \rightarrow \mathbb{R}$ 
defined by
\begin{displaymath}
\rho(r) = \max_{w \in \mathcal{W}} \{ \rho_{S}(r,w) - \tau \mu_{\Delta}(w) \}
\end{displaymath}
\begin{displaymath}
\mu(\hat{v}) = \min_{u \in \mathcal{U}} \{ \mu_{S}(\hat{v},u) - \tau \rho_{\Delta}(u) \}
\end{displaymath}
for any scalar parameter $\tau > 0 $.
\end{corollary}

It follows from Theorem \ref{thm:SGT} and Corollary \ref{cor:Scales} that the objective in (\ref{eq:AuxPerfObject}) can be achieved by designing a switching law 
$\varphi:\mathcal{Q} \rightarrow \mathcal{U}$ such that the feedback interconnection $S=(M,\varphi)$ satisfies (\ref{eq:SGT1}) with
\begin{eqnarray}
\rho_{S}(r,w) & = & \tau \mu(w) -R, \\ \nonumber
\mu_{S}(\hat{v},u) & = & \hat{v} + \tau \gamma_o \rho(u) \nonumber
\end{eqnarray}
for some $R > 0$ and $\tau >0$, for $\rho$ and $\mu$ defined as in (\ref{eq:gainofDelta}).
Note that the exogenous input $r$ can be assumed to be constant here (representing the desired rate of convergence).

With the auxiliary robust performance objective reformulated as a design objective for system $S = (M,\varphi)$, 
the design of the full state feedback switching law can be carried out by solving a min-max optimization problem using 
standard dynamic programming techniques \cite{BOOK:B1,BOOK:B2}. 
The following theorem precisely formulates this approach: In this setting, function $J$ and operator $\mathbb{T}$ 
play the role of the ``cost-to-go" function and the ``dynamic programming" operator, respectively.
Value iteration is used to solve for the cost-to-go function, 
and the sought full state feedback switching law is then simply the optimizing argument.  
For the sake of completeness, a proof is outlined in the Appendix.

\begin{theorem}
\label{thm:DP}
Consider a deterministic finite state machine $M$ with state transition equation
\begin{displaymath}
q(t+1)=f(q(t),u(t),w(t))
\end{displaymath} 
where $\mathcal{Q}$ is the state set and $\mathcal{U}$ and $\mathcal{W}$ are the input sets.
Let $\sigma : \mathcal{Q} \times \mathcal{U} \times \mathcal{W} \rightarrow \mathbb{R}$ be a given function.
The following three statements are equivalent:
\begin{itemize}
\item[(a)] There exists a $\varphi:\mathcal{Q} \rightarrow \mathcal{U}$ such that the closed loop system $(M,\varphi)$ satisfies
\begin{equation}
\label{eq:SigmaObjective}
\inf_{T \geq 0} \sum_{t=0}^{T} \sigma(q(t),\varphi(q(t)),w(t)) > -\infty.
\end{equation}
\item[(b)] There exists a function $J:\mathcal{Q} \rightarrow \mathbb{R}_+$ such that the inequality  
\begin{equation}
\label{eq:BellmanInequality}
J(q) \geq \mathbb{T}(J(q))
\end{equation}
holds for any $q \in \mathcal{Q}$, for $\mathbb{T}: \mathbb{R}^{\mathcal{Q}} \rightarrow \mathbb{R}^{\mathcal{Q}}$ defined by
\begin{equation}
\label{eq:Toperator}
\mathbb{T}(J(q)) = \min_{u \in \mathcal{U}} \max_{w \in \mathcal{W}} \{-\sigma(q,u,w) + J(f(q,u,w)) \}.
\end{equation}
\item[(c)] The sequence of functions $J_k : \mathcal{Q} \rightarrow \mathbb{R}$, $k \in \mathbb{Z}_+$, defined recursively by
\begin{eqnarray}
\label{eq:Jsequence}
\nonumber
J_0 & = & 0 \\ 
J_{k+1} & = & \max \{0,\mathbb{T}(J_k)\}
\end{eqnarray}
converges. 
\end{itemize}
\end{theorem}

For the particular control problem of interest, the cost function 
\begin{displaymath}
\sigma(q,u,w) = \tau \mu(w) - R - \hat{v}(q,u) - \tau \gamma_o \rho(u)
\end{displaymath}
involves two parameters, the convergence rate $R$ and the small gain scaling parameter $\tau$. 
Ultimately, the goal is to maximize $R>0$ for which there exists a $J: \mathcal{Q} \rightarrow \mathbb{R}$ and 
a $\tau > 0$ such that (\ref{eq:BellmanInequality}) holds. 
However, since it is not possible to directly compute the optimal value of $R$, a numerical search is carried out yielding a suboptimal value of $R$: 
First, a range of values of $\tau$ for which the design objective can be met when $R=0$ is computed. 
Different values of $\tau$ are then sampled in this range to compute the largest feasible corresponding value of $R$, 
with the largest of those being the (suboptimal) guaranteed rate of convergence.

\section{Illustrative Examples}
\label{sec:Examples}

An academic example is presented in this section to illustrate the proposed design procedure. 
The example was chosen for its simplicity and amenability to analysis in the ideal (full state feedback) case, 
thus allowing us to numerically compare the performance of our finite state controllers in the binary sensing setting 
to the best achievable performance in the ideal setting for a given pair of plants.

Consider a harmonic oscillator
\begin{eqnarray} \nonumber
\dot{x}_1 & = & x_2 \\ \nonumber
\dot{x}_2 & = & k x_1
\end{eqnarray}
described in polar coordinates by
\begin{eqnarray} \nonumber
\dot{r} & = & (1+k)r\sin(\theta) \cos(\theta) \\ \nonumber
\dot{\theta} & = & (1+k)\cos^2(\theta) - 1
\end{eqnarray}
When $k=-1$, $\dot{r}=0$ and the state trajectories are concentric circles centered at the origin. 
For any $k \neq -1$, $\dot{r} <0$ in exactly two sectors of the state space given by
\begin{displaymath} 
\theta \in \left\{ \begin{array}{cc}
(0,\pi/2) \cup (\pi,3 \pi/2) & \textrm{      when   } k <-1 \\
(\pi/2,\pi) \cup (3 \pi/2, 2\pi) & \textrm{      when   } k > -1
\end{array} \right.
\end{displaymath}
Thus in the ideal case where the switching controller has full access to the state and switching can occur at any time, 
it is always possible to stabilize a pair of harmonic oscillators with $k=-1$ and $k \in \mathbb{R}\setminus \{-1\}$ by appropriately switching between them.

Now suppose that the system is sampled and the only sensor information available for use by the controller is the sign of the position measurement. 
The stabilization problem becomes much more difficult because of the non-trivial state estimation problem that arises due to binary sensing and 
the time delays introduced by sampling. 
The approach described in this paper will be used to design a finite state stabilizing controller. 
 
The dynamics of the system, sampled at times $tT_S$, $t \in \mathbb{Z}_+$, for a given sampling period $T_S$, are
\begin{eqnarray}
\label{eq:systemP}
\nonumber
x(t+1) & = & A_S(u(t)) x(t) \\  \nonumber
y(t) & = & sgn(x_1(t))\\ \nonumber
v(t) & = & log \Big( \frac{\| x(t+1) \|_{2}}{\| x(t) \|_{2}} \Big) 
\end{eqnarray}
where $A_S(0)$, $A_S(1)$ are given by $A_S(u)=e^{A(u)T_S}$ with
\begin{displaymath}
A(0) = \left[ \begin{array} {ll}
0 & 1 \\
-1 & 0 
\end{array} \right],
\end{displaymath}
\begin{displaymath}
A(1) = \left[ \begin{array} {ll}
0 & 1 \\
k_o & 0 
\end{array} \right], \textrm{       } k_o \in \mathbb{R} \setminus \{-1\}.
\end{displaymath}
In this setting, applying inputs $u=0$ and $u=1$ effectively corresponds to switching between ``passive" and ``aggressive" control modes, respectively.
The sensor measurement $y$ is assumed to be available at the beginning of each sampling interval $[tT_S,(t+1)T_S)$
at which point a plant (or equivalently, a value of $k$) is chosen by the controller $K$ and held until the end of the sampling interval. 
The sampling interval $T_S$ is assumed to be a design parameter in this case.
In order to counteract some of the conservatism introduced in the computation of an upper bound for the gain of $\Delta$, the unit circle is 
uniformly partitioned into $2n$ intervals and the sampling rate is matched to this partition, that is $T_S = \frac{ \pi}{n}$. 
This ensures that a partition interval maps to another partition interval in the "passive" control mode. 

\begin{center}
   \begin{table}
   \label{Table:Parameters}
   \caption{Explanation of Reported Parameters.}
   \begin{tabular}{lcl}
   \hline
   \noalign{\smallskip}
   $n$ &  & 2$n$ is the total number of intervals partitioning the unit circle \\
   $N$ &  & \textit{Actual} number of states of the deterministic finite state machine approximation ($N \leq 2n(2n-1)+1$) \\
   $T_S$ &  & Sampling interval \\
   $\gamma$ &  & Tightest verifiable gain bound for the approximation error \\
   $R$ &  & Guaranteed rate of exponential convergence of the sampled system \\
   $p$ &  & Number of iterations required for convergence of the Value Iteration algorithm \\
   \hline
   \end{tabular}
   \end{table}
   \end{center}

\vspace*{20pt}
\begin{example}
\label{eg:1}
In particular, consider the case where $k_o=-3$. 
The coarsest partition of the unit circle that gives rise to a good enough approximate model allowing for successful control design corresponds to $n=5$. 
The relevant data for several choices of $n$ are shown in Table II, and the implementation of the resulting controllers 
is plotted in Figure \ref{Fig:Eg1}, together with the corresponding plot in the ``ideal" (full state feedback, unsampled)
setting  for comparison. 
Table I explains the significance of the various parameters reported for the examples. 
In this case, refining the partition by increasing $n$ improves both the provable gain bound for $\Delta$ and the rate of convergence up to a certain point.
The improvement seems to taper off beyond $n=15$, likely due to numerical errors.
Note that while $R$ decreases for the cases where $n=15$ and $n=20$, the actual rate of convergence still improves as seen 
in the simulations, thus reinforcing the fact that the actual rate of convergence may be significantly better than the provable rate.
\end{example}

 \begin{center}
   \begin{table}
   \label{Examplekminus3}
   \caption{Data for Example \ref{eg:1}.}
   \begin{tabular}{lllllll}
   \hline\noalign{\smallskip}
   $n$ \hspace{0.85in} & $N$ \hspace{0.85in} & $T_S$ \hspace{0.85in} & $\gamma$ \hspace{0.85in} & $R$  \hspace{0.85in} & $p$  \\
   \noalign{\smallskip}
   \hline
   \noalign{\smallskip}
   5 & 39 & 0.6283 & 0.75 & 0.0160 & 10 \\
   10 & 107 & 0.3142 & 0.625 & 0.0267 & 20 \\
   15 & 207 & 0.2094 & 0.5833 & 0.0195 & 28 \\
   20 & 331& 0.1571 & 0.5625 & 0.0152 & 39 \\
   \hline
   \end{tabular}
   \end{table}
   \end{center}

   \begin{figure}[thpb]
      \centerline{\psfig{file=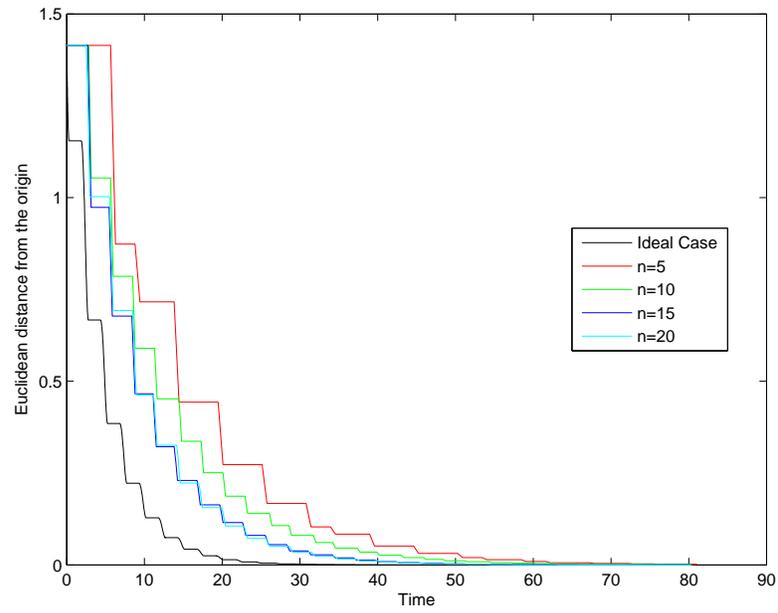,height=9cm}}
      \caption{Implementation of the DFM controller in Example 1}
      \label{Fig:Eg1}
   \end{figure}

\vspace*{20pt}
\begin{example}
\label{eg:2}
Now consider the case where $k_o=2$. 
In this case, the coarsest workable partition of the unit circle giving rise to a good enough approximate model allowing for 
successful control design corresponds to $n=6$.  
The relevant data for several choices of $n$ are shown in Table III, and the implementation of the resulting controllers 
is plotted in Figure \ref{Fig:Eg1-2}, together with the corresponding plot in the ``ideal" setting for comparison. 
Note in this case that even though the best provable gain bound for the error system $\Delta$ remains unchanged between $n=6$ and $n=7$, 
the difference between $\hat{v}$ and $v$ decreases as the unit circle partition is refined, hence the quality of approximation effectively 
improves as does the guaranteed rate of convergence.
Also note that in this case the improvement stalls much sooner, around $n=8$.
This is likely due to the fact that the ``aggressive" control law is much more aggressive than in the previous example, 
and hence numerical errors affect the computation sooner.
\end{example}

  \begin{center}
   \begin{table}
   \label{Examplek2}
   \caption{Data for Example \ref{eg:2}.}
   \begin{tabular}{lllllll}
   \hline\noalign{\smallskip}
   $n$ \hspace{0.85in} & $N$ \hspace{0.85in} & $T_S$ \hspace{0.85in} & $\gamma$ \hspace{0.85in} & $R$ \hspace{0.85in} &  $p$  \\
   \noalign{\smallskip}
   \hline
   \noalign{\smallskip}
   6 & 39 & 0.5236 & 1 & 0.0141 & 11 \\
   7 & 55 & 0.4488 & 1 & 0.0387 & 13 \\
   8 & 69 & 0.3927 & 1 & 0.0279 & 17 \\
      \hline
   \end{tabular}
   \end{table}
   \end{center} 
  
   \begin{figure}[thpb]
      \centerline{\psfig{file=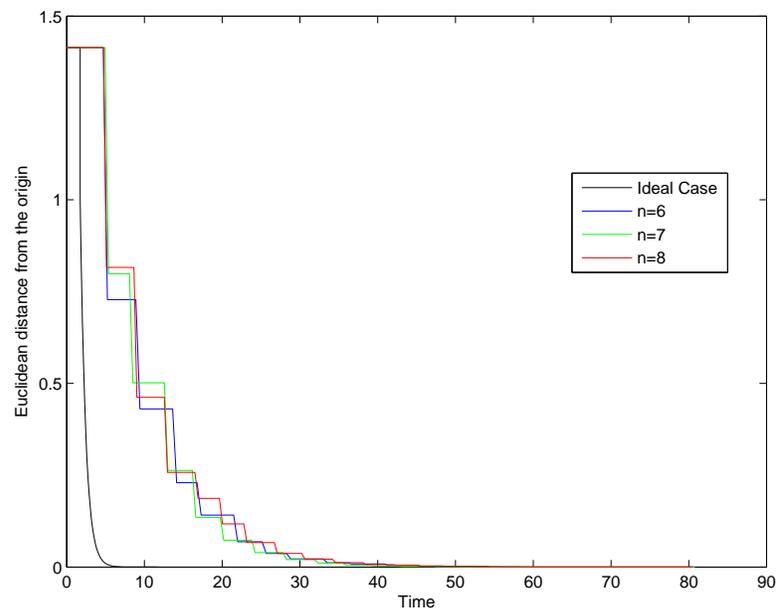,height=9cm}}
      \caption{Implementation of the DFM controller in Example \ref{eg:2}.}
      \label{Fig:Eg1-2}
   \end{figure}

\section{CONCLUSIONS AND FUTURE WORK}
\label{sec:Conclusions}

A new constructive approach for designing switching controllers to stabilize pairs of discrete-time homogenous second order 
systems under limited information (due to binary sensors) was presented. 
The hybrid plant is first approximated by a deterministic finite state machine, and a
useable description of the resulting approximation error, in the form of a "gain" bound, is established. 
Using a "small gain" argument and dynamic programming techniques, a controller is then synthesized to robustly 
stabilize the nominal (approximate) model in the presence of admissible (approximation) uncertainty. 
The resulting stabilizing controller then consists of a deterministic finite state machine observer for the hybrid plant and 
a corresponding full state feedback switching law. 
While the proposed approach is inspired from classical robust control, the use of deterministic finite state machines as an alternative class of nominal models 
necessitates the development of a new set of complementary approximation, analysis and control tools.  

Future work will focus on three main directions:
\begin{enumerate} 
\item Reducing the conservatism of the approach: As mentioned in the paper, the proposed approach for computing gain bounds for the 
approximation uncertainty $\Delta$ is efficient (polynomial in the size of the nominal models) but conservative. 
One direction of future effort will thus focus on developing alternative tractable approaches for verifying tighter gain bounds, 
allowing for the design of stabilizing controllers with better performance. 
\item Extending the approach to broader problems: Another direction of future work will look into extending the proposed analysis, approximation 
and synthesis tools to broader classes of plants beyond homogeneous systems, in addition to identifying alternative performance objectives 
(beyond stability) that can be captured by "gain conditions". 
\item Addressing the issue of scalability: While the design approach presented here can, in principle, be applied to systems of arbitrary order, 
it scales poorly (exponentially) with plant order, as do most existing hybrid design techniques. 
Thus one important direction of future work is finding ways of improving the scalability of the approach, potentially by identifying and exploiting system structure.
\end{enumerate}

\section{Acknowledgments}

The first author was supported by AFOSR Grant FA9550-04-1-0052 while she was at MIT
and by startup funds from the Whiting School of Engineering while at The Johns Hopkins University.

\appendices

\section{Proof of Theorem \ref{thm:DP} } 

Several statements will be useful in proving Theorem \ref{thm:DP}. 
The first one, Theorem \ref{thm:Vstability}, is adapted from Theorem 3 in \cite{JOUR:TMD} and is presented here without proof.

\begin{theorem}
\label{thm:Vstability}
Consider a deterministic finite state machine $M$ with state transition equation
\begin{displaymath}
q(t+1)=f(q(t),u(t),w(t))
\end{displaymath} 
where $\mathcal{Q}$ is the state set and $\mathcal{U}$ and $\mathcal{W}$ are the input sets. 
Let $\sigma : \mathcal{Q} \times \mathcal{U} \times \mathcal{W} \rightarrow \mathbb{R}$ be a given function. 
The following two statements are equivalent:
\begin{itemize} 
\item[(a)] The inequality
\begin{displaymath}
\inf_{T \geq 0 }\sum_{t=0}^{T} \sigma(q(t),u(t),w(t)) > - \infty
\end{displaymath}
is satisfied for any $(\mathbf{u},\mathbf{w}) \in \mathcal{U}^{\mathbb{Z}_+} \times \mathcal{W}^{\mathbb{Z}_+}$ and $q(0) \in \mathcal{Q}$.
\item[(b)] There exists a non-negative function $ J : \mathcal{Q} \rightarrow \mathbb{R}_+$ such that the inequality
\begin{displaymath}
J(f(q,u,w)) - J(q) \leq \sigma(q,u,w)
\end{displaymath}
holds for all $q \in \mathcal{Q}$, $u \in \mathcal{U}$ and $w \in \mathcal{W}$.
\end{itemize} 
\end{theorem}

In the terminology of Willem's theory of dissipative systems \cite{JOUR:Willem1972}, 
$J$ in Theorem \ref{thm:Vstability} is the storage function of the dissipative system with supply rate $\sigma$.

\vspace{2mm}
\begin{lemma}
\label{lemma:Tmonotonic}
Function $\mathbb{T}$ defined in (\ref{eq:Toperator}) is monotonic, that is 
\begin{displaymath}
J_{1} \leq J_{2} \Rightarrow \mathbb{T}(J_{1}) \leq \mathbb{T}(J_{2}).
\end{displaymath}
\end{lemma}

\begin{proof}
When $J_{1} \leq J_{2}$, we have
\begin{eqnarray*} 
& & J_1(q) \leq J_2(q), \textrm{   } \forall q \in \mathcal{Q} \\ 
& \Rightarrow & \sigma(q,u,w) + J_{1}(f(q,u,w)) \leq \sigma(q,u,w) + J_{2}(f(q,u,w)), \forall q \in \mathcal{Q}, u \in \mathcal{U}, w \in \mathcal{W} \\ 
& \Rightarrow & \max_{w \in \mathcal{W}} \{ \sigma(q,u,w) + J_{1}(f(q,u,w)) \} \leq \max_{w \in \mathcal{W}} \{ \sigma(q,u,w) + J_{2}(f(q,u,w)) \}, \forall q \in \mathcal{Q}, u \in \mathcal{U} \\ 
& \Rightarrow & \min_{u \in \mathcal{U}} \max_{w \in \mathcal{W}} \{ \sigma(q,u,w) + J_{1}(f(q,u,w)) \} \leq \min_{u \in \mathcal{U}} \max_{w \in \mathcal{W}} \{ \sigma(q,u,w) + J_{2}(f(q,u,w)) \}, \forall q \in \mathcal{Q} \\ 
& \Rightarrow & \mathbb{T}(J_{1})(q) \leq \mathbb{T}(J_{2})(q), \forall q \in \mathcal{Q} 
\end{eqnarray*}
Thus $\mathbb{T}(J_1) \leq \mathbb{T}(J_2)$.
\end{proof}

\begin{lemma}
\label{lemma:Jmonotonic}
The sequence $\{J_{k}\}$ defined in (\ref{eq:Jsequence}) is monotonically increasing.
\end{lemma}

\begin{proof}
The proof is by induction on $k$. We have
\begin{displaymath}
J_{1} = \max \{0,\mathbb{T}(J_{0})\} \geq 0 = J_{0}.
\end{displaymath}
Now suppose $J_{k} \geq J_{k-1}$. Then
\begin{eqnarray} \nonumber
J_{k+1} & = & \max \{0, \mathbb{T}(J_{k})\} \\ \nonumber
& \geq & \max \{ 0, \mathbb{T}(J_{k-1}) \} \\ \nonumber 
& \geq & J_{k}
\end{eqnarray}
which completes the proof. 
\end{proof}

Given a scalar $c \in \mathbb{R}$ and a function  $J: \mathcal{Q} \rightarrow \mathbb{R}$, 
denote by $J + c$ the function from $\mathcal{Q}$ to $\mathbb{R}$ defined by $(J+c)(q) = J(q)+c$. 

\begin{lemma}
\label{lemma:Jlinearity}
For any $J: \mathcal{Q} \rightarrow \mathbb{R}$ and $c \in \mathbb{R}$, $\mathbb{T}(J+c)=\mathbb{T}(J)+c$.
\end{lemma}

\begin{proof}
For any $q \in \mathcal{Q}$ we have
\begin{eqnarray}
\nonumber
\mathbb{T}((J+c)(q)) & = & \min_{u \in \mathcal{U}} \max_{w \in \mathcal{W}} \{ \sigma(q,u,w) + (J+c)(f(q,u,w)) \} \\ \nonumber
& = & \min_{u \in \mathcal{U}} \max_{w \in \mathcal{W}} \{ \sigma(q,u,w) + J(f(q,u,w)) + c \} \\ \nonumber
& = & \min_{u \in \mathcal{U}} \max_{w \in \mathcal{W}} \{ \sigma(q,u,w) + J(f(q,u,w)) \} + c \\ \nonumber
& = & \mathbb{T}(J(q)) + c
\end{eqnarray}
\end{proof}

We are now ready to prove Theorem \ref{thm:DP}.

\begin{proof}
\textit{(Theorem \ref{thm:DP})} \\
(b) $\Rightarrow$ (a): Suppose there exists a function $J: \mathcal{Q} \rightarrow \mathbb{R}$ 
satisfying (\ref{eq:BellmanInequality}), and let
\begin{displaymath}
\varphi(q) = \arg\min_{u \in \mathcal{U}} \Big\{ \max_{w \in \mathcal{W}} \Big( - \sigma(q,u,w) + J(f(q,u,w)) \Big) \Big\}.
\end{displaymath}
For any $q \in \mathcal{Q}$, we have
\begin{eqnarray} \nonumber 
J(q) & \geq & \mathbb{T}(J(q)) \\ \nonumber
& = & \min_{u} \max_{w \in \mathcal{W}} \{ - \sigma(q,u,w) + J(f(q,u,w)) \} \\ \nonumber
& = & \max_{w \in \mathcal{W}} \{ - \sigma(q,\varphi(q),w) + J(f(q,\varphi(q),w)) \} \\ \nonumber
& \geq & - \sigma(q,\varphi(q),w) + J(f(q,\varphi(q),w)), \forall w \in \mathcal{W}.
\end{eqnarray}
It follows from Theorem \ref{thm:Vstability} that the (deterministic finite state machine) closed loop system 
$(M,\varphi)$ satisfies (\ref{eq:SigmaObjective}). 

(a) $\Rightarrow$ (c): Suppose there exists a $\varphi$ such that the closed loop system satisfies (\ref{eq:SigmaObjective}). 
By Theorem \ref{thm:Vstability}, there exists a $J : \mathcal{Q} \rightarrow \mathbb{R}_+$ such that
\begin{displaymath}
J(f(q,\varphi(q),w)) - J(q) \leq \sigma(q,\varphi(q),w), \forall q \in \mathcal{Q}, w \in \mathcal{W}  
\end{displaymath}
We then have
\begin{eqnarray}
\nonumber
J(q) & \geq &  - \sigma(q,\varphi(q),w) + J(f(q,\varphi(q),w)), \forall q \in \mathcal{Q},w \in \mathcal{W} \\ \nonumber
& \geq & \max_{w \in \mathcal{W}} \{ - \sigma(q,\varphi(q),w) + J(f(q,\varphi(q),w)) \}, \forall q \in \mathcal{Q} \\ \nonumber
& = & \min_{u \in \mathcal{U}} \max_{w \in \mathcal{W}} \{ - \sigma(q,u,w) + J(f(q,u,w)) \}, \forall q \in \mathcal{Q} \\ \nonumber
& = & \mathbb{T}(J(q)), \forall q \in \mathcal{Q}
\end{eqnarray}

It follows that function $J+c$ also satisfies
\begin{displaymath}
(J+c) \geq \mathbb{T}((J+c))
\end{displaymath}
for any choice of $c \in \mathbb{R}$. 
Since the set $\mathcal{Q}$ is finite, we can assume, without loss of generality, that $J \geq 0$ with $\displaystyle \min_{q}J(q)=0$ 
and $\displaystyle \max_{q}J(q)=m \geq 0$. 

Moreover, the sequence $\{J_{k}\}$ defined in (\ref{eq:Jsequence}) is bounded above by $J$. The proof is by induction on $k$. We have
\begin{displaymath}
J_{0} = 0 \leq J
\end{displaymath}
Suppose that $J_{k} \leq J$. Then
\begin{displaymath}
\mathbb{T}(J_{k}) \leq \mathbb{T}(J) \leq J
\end{displaymath}
where the first inequality follows from Lemma \ref{lemma:Tmonotonic} and
\begin{displaymath}
J_{k+1} = \max \{0,\mathbb{T}(J_{k})\} \leq J.
\end{displaymath}
Sequence $\{J_{k}\}$ is thus monotonically increasing (Lemma \ref{lemma:Jmonotonic}) and bounded above by $J$. 
Hence it converges to $\displaystyle J^{*} = \lim_{k \rightarrow \infty} J_{k} \leq J$.

(c) $\Rightarrow$ (b): Suppose that the sequence $\{J_{k}\}$ converges to $\displaystyle J^{*} = \lim_{k \rightarrow \infty} J_{k}$ 
and let 
\begin{displaymath}
\epsilon_k = \max_{q \in \mathcal{Q}} \{ J^*(q) - J_k(q) \}.
\end{displaymath} 
Note that $\epsilon_k \geq 0$, the sequence $\{ \epsilon_k \}$ is monotonically decreasing (follows from Lemma \ref{lemma:Jmonotonic}) 
and $\displaystyle \lim_{k \rightarrow \infty} \epsilon_k = 0$.
Moreover
\begin{displaymath}
J_{k}(q) \geq J^{*}(q) - \epsilon_{k}, \forall q \in \mathcal{Q}.
\end{displaymath}
It follows from Lemmas \ref{lemma:Tmonotonic}  and \ref{lemma:Jlinearity} that
\begin{displaymath}
\mathbb{T}(J_{k}) \geq \mathbb{T}(J^{*} - \epsilon_{k}) = \mathbb{T}(J^{*}) - \epsilon_{k}.
\end{displaymath}
But we have 
\begin{eqnarray}
\nonumber
J_{k+1} & = & \max{\{0,\mathbb{T}(J_k)\}} \\ \nonumber
& \geq & \mathbb{T}(J_{k}) \\ \nonumber
& \geq & \mathbb{T}(J^{*}) - \epsilon_{k}.
\end{eqnarray}
Thus
\begin{displaymath}
J^* = \lim_{k \rightarrow \infty} J_{k+1} \geq \mathbb{T}(J^{*}) - \lim_{k \rightarrow \infty} \epsilon_{k} = \mathbb{T}(J^*).
\end{displaymath}
\end{proof}


\bibliographystyle{IEEEtranS}
\bibliography{References}

\end{document}